\documentclass[11pt]{amsart}
\usepackage{amsthm}
\usepackage{latexsym}
\usepackage[dvips]{graphicx}
\usepackage[dvips]{color}
\usepackage{a4wide}

\theoremstyle{plain}
\newtheorem*{theorem*}{Theorem}
\newtheorem*{lemma*} {Lemma}
\newtheorem*{corollary*} {Corollary}
\newtheorem*{proposition*}{Proposition}
\newtheorem*{conjecture*}{Conjecture}
\newtheorem{theorem}{Theorem}[section]
\newtheorem{lemma}[theorem]{Lemma}

\theoremstyle{remark}

\newtheorem*{definition}{Definition}

\theoremstyle{definition}

\begin{document}

\title [Disk surgery and the primitive disk complexes]
{Disk surgery and the primitive disk complexes of the $3$-sphere}

\author[S. Cho]{Sangbum Cho}
\address{Department of Mathematics Education,
Hanyang University, Seoul 04763, Korea}
\email{scho@hanyang.ac.kr}

\author[Y. Koda]{Yuya Koda}
\address{Department of Mathematics,
Hiroshima University, 1-3-1 Kagamiyama,
Higashi-Hiroshima, 739-8526, Japan}
\email{ykoda@hiroshima-u.ac.jp}

\author[J. H. Lee]{Jung Hoon Lee}
\address{Department of Mathematics and Institute of Pure and Applied Mathematics,
Chonbuk National University, Jeonju 54896, Korea}
\email{junghoon@jbnu.ac.kr}


\begin{abstract}
Given a genus-$g$ Heegaard splitting of the $3$-sphere with $g \ge 3$,
we show that the primitive disk complex for the splitting is not weakly closed under disk surgery operation.
That is, there exist two primitive disks in one of the handlebodies of the splitting such that any disk surgery on one along the other one yields no primitive disks.
\end{abstract}

\maketitle

\section{Introduction}\label{sec1}

It is well known that any closed orientable $3$-manifold can be decomposed into two handlebodies $V$ and $W$ of the same genus $g$,
which we call a {\it genus-$g$ Heegaard splitting} of the manifold.
We denote the splitting by the triple $(V, W; \Sigma)$ where $\Sigma = \partial V = \partial W$ is a closed orientable surface,
called a {\it Heegaard surface}, of genus $g$.
In particular, the $3$-sphere admits a Heegaard splitting of each genus $g \ge 0$, and
it was shown in \cite{Waldhausen} that the splitting is unique up to isotopy for each genus.
A Heegaard splitting $(V,W; \Sigma)$ of a $3$-manifold $M$ is said to be {\it stabilized}
if there exists essential disks $D$ and $\overline{D}$ in $V$ and $W$ respectively such that
$\partial D$ intersects $\partial \overline{D}$ transversely in a single point.
A $3$-manifold $M$ admits a stabilized Heegaard splitting of genus-$2$ if and only if
$M$ is one of the $3$-sphere, $S^2 \times S^1$ or a lens space $L(p, q)$.

For a handlebody $V$ of genus $g \ge 2$, the {\it disk complex} $\mathcal{K}(V)$ of $V$ is the simplicial complex defined as follows.
The vertices are the isotopy classes of compressing disks in $V$, and
a collection of distinct $k+1$ vertices spans a $k$-simplex if the vertices are represented by pairwise disjoint disks.
The disk complex $\mathcal{K}(V)$ is $(3g - 4)$-dimensional and is not locally finite.
When the handlebody $V$ is one of the handlebodies of a stabilized genus-$g$ Heegaard splitting $(V, W; \Sigma)$, with $g \geq 2$,
the disk complex $\mathcal{K}(V)$ has a special kind of subcomplex, called the {\it primitive disk complex}.
The primitive disk complex, denoted by $\mathcal P(V)$, is the full subcomplex of $\mathcal{K}(V)$ spanned by the vertices of {\it primitive disks}.
A compressing disk $D$ in $V$ is called {\it primitive} if
there exists a compressing disk $\overline{D}$ in $W$ such that
$\partial D$ intersects $\partial \overline{D}$ transversely in a single point.
We call such a disk $\overline{D}$ a {\it dual disk} of $D$.

For the genus-$2$ Heegaard splitting $(V, W; \Sigma)$ of each of the $3$-sphere, $S^2 \times S^1$ and lens spaces $L(p, q)$,
the structure of the primitive disk complex $\mathcal P(V)$ is fully studied in \cite{Cho08}, \cite{Cho13}, \cite{CK14}, \cite{CK16} and \cite{CK18}.
Understanding the structure of the primitive disk complexes enables us to obtain finite presentations of the mapping class groups of the splittings
by investigating the simplicial action of the group on the primitive disk complex.
Actually, it was shown that the primitive disk complex $\mathcal P(V)$ is contractible
for the genus-$2$ Heegaard splitting of each of the $3$-sphere, $S^2 \times S^1$ and some lens spaces.
Furthermore, the quotient of $\mathcal P(V)$ by the action of the mapping class group of the splitting is a simple finite complex for each case, and
the group is presented easily in terms of the isotropy subgroups of the simplices of the quotient.

The contractibility of $\mathcal P(V)$ in the case of the genus-$2$ splittings is based on the fact that
$\mathcal P(V)$ is {\it closed under the disk surgery operation}.
In other words, given any two primitive disks in $V$ intersecting each other,
any surgery on one disk along the other one always yields a primitive disk, whose meaning explained in detail in the next section.
In particular, it was shown in \cite{Cho08} that the primitive disk complex for the genus-$2$ Heegaard splitting of the $3$-sphere is closed under disk surgery operation, and
so it has been conjectured that it is also true for the higher genus splittings of the $3$-sphere.
The main result of this work is to show that it is not true.
In fact, we show further that, in the case of genus $g \geq 3$, there exist two primitive disks such that ``no'' surgery on one along the other one yields a primitive disk.
The main result is stated as follows.

\begin{theorem}
\label{thm: main theorem}
Let $(V, W; \Sigma)$ be a genus-$g$ Heegaard splitting of the $3$-sphere with $g \ge 3$.
Then the primitive disk complex $\mathcal P(V)$ is not closed under the disk surgery operation.
In fact, $\mathcal P(V)$ is not even weakly closed under the disk surgery operation.
\end{theorem}

Throughout the paper, any disks (except subdisks of a disk) in a 3-manifold are always assumed to be properly embedded, and
their intersection is transverse and minimal up to isotopy.
In particular, if a disk $D$ intersects a disk $E$, then $D \cap E$ is a collection of pairwise disjoint arcs that are properly embedded in both $D$ and $E$.
For convenience, we will not distinguish disks from their isotopy classes in their notation.

\section{Disk surgery operation}
\label{sec: Surgery paths in disk complexes}

Let $M$ be a compact, orientable, irreducible $3$-manifold with compressible boundary.
The {\it disk complex} $\mathcal{K}(M)$ for $M$ is a simplicial complex defined as follows.
The vertices are the isotopy classes of compressing disks in $M$, and
a collection of distinct $k+1$ vertices spans a $k$-simplex if and only if the vertices are represented by pairwise disjoint disks.

Let $D$ and $E$ be compressing disks in $M$ with $D \cap E \neq \emptyset$.
Let $\Delta$ be a disk cut off from $E$ by an outermost arc $\delta$ of $D \cap E$ in $E$ such that $\Delta \cap D = \delta$.
We call such a subdisk $\Delta$ an {\it outermost subdisk} of $E$ cut off by $D \cap E$.
The arc $\delta$ cuts $D$ into two subdisks, say $C_1$ and $C_2$.
Let $D_1 = C_1 \cup \Delta$ and $D_2 = C_2 \cup \Delta$.
By a slight isotopy, the two disks $D_1$ and $D_2$ can be moved to be disjoint from $D$.
We say that $D_1$ and $D_2$ are the disks obtained by {\it surgery} on $D$ along $E$ (with the outermost subdisk $\Delta$).
Of course there are many choices of the outermost subdisk of $E$ cut off by $D \cap E$, and the resulting two disks from surgery depend on the choice of the outermost subdisks.
We note that each of $D_1$ and $D_2$ has fewer arcs of intersection with $E$ than $D$ had since at least the arc $\delta$ no longer counts.
Further, if $D$ is non-separating, at least one of $D_1$ and $D_2$ is non-separating.

\begin{definition}
\label{def: convexity}
Let $\mathcal X$ be a full subcomplex of $\mathcal{K}(M)$.
\begin{enumerate}
\item
We say that $\mathcal X$ is {\it closed under disk surgery operation}
if for any disks $D$ and $E$ with $D \cap E \neq \emptyset$ representing vertices of $\mathcal X$,
every surgery on $D$ along $E$ always yields a disk representing a vertex of $\mathcal X$.
\item
We say that $\mathcal X$ is {\it weakly closed under disk surgery operation}
if for any disks $D$ and $E$ with $D \cap E \neq \emptyset$ representing vertices of $\mathcal X$,
there exists a surgery on $D$ along $E$ which yields a disk representing a vertex of $\mathcal X$.
\end{enumerate}
\end{definition}

It is clear that the ``closedness'' implies the ``weak closedness''.
For the weak closedness, it is enough to find only an outermost subdisk $\Delta$ of $E$ such that
at least one of the two disks obtained from surgery on $D$ along $E$ with $\Delta$ yields a disk representing a vertex of $\mathcal X$,
while for the closedness, we need to show that the surgery with ``any'' outermost subdisk always yields a disk representing a vertex of $\mathcal X$.

It is easy to see that the disk complex $\mathcal K(M)$ itself is closed under disk surgery operation, and
so is the {\it non-separating disk complex}, denoted by $\mathcal D(M)$, the full subcomplex of $\mathcal K(M)$ spanned by all vertices of non-separating disks.
The weak closedness with the closedness have served as a useful tool to understand the structure of various subcomplexes of the disk complex, for example we have the following.

\begin{theorem}
\label{thm:convexity implies contractibility}
Let $\mathcal X$ be a full subcomplex of $\mathcal{K}(M)$.
\begin{enumerate}
   \item If $\mathcal X$ is weakly closed under disk surgery operation, then $\mathcal X$ is connected.
   \item If $\mathcal X$ is closed under disk surgery operation, then $\mathcal X$ is contractible.
\end{enumerate}
\end{theorem}

The first statement of the theorem is easy to verify.
Whenever we have two vertices $D$ and $E$ of $\mathcal X$ far from each other, that is, $D \cap E \neq \emptyset$, then
we have an outermost subdisk $\Delta$ of $E$ cut off by $D \cap E$ such that
the surgery on $D$ along $E$ with $\Delta$ yields a disk, say $D_1$, representing a vertex of $\mathcal X$.
The vertex of $D_1$ is joined by an edge to $D$.
If $D_1 \cap E \neq \emptyset$, we do surgery on $D_1$ along $E$ to have a vertex of $\mathcal X$ and so on.
Then eventually we have a path in $\mathcal X$ from $D$ to $E$.
The second statement was essentially proved in \cite{McCullough} and updated in \cite{Cho08}.
In \cite{Cho08}, the contractibility is proved in the case where $M$ is a handlebody, but
the proof is still valid for an arbitrary irreducible manifold with compressible boundary.

From Theorem \ref{thm:convexity implies contractibility}, we see that the disk complex $\mathcal K(M)$ and the non-separating disk complex $\mathcal D(M)$ are all contractible.
Recall that when a handlebody $V$ is one of the handlebodies of the genus-$g$ Heegaard splitting $(V, W; \Sigma)$, with $g \geq 2$, of the $3$-sphere, $S^2 \times S^1$ or a lens space $L(p, q)$, the primitive disk complex $\mathcal P(V)$ is the full subcomplex of $\mathcal{K}(V)$ spanned by the vertices of primitive disks in $V$.
The followings are known results on the primitive disk complexes $\mathcal P(V)$ for the genus-$2$ splittings (see \cite{Cho08, Cho13, CK14, CK16}).

\begin{enumerate}
\item
For the genus-$2$ splittings of $3$-sphere and $S^2 \times S^1$, the complex $\mathcal P(V)$ is closed under disk surgery operation, and hence they are all contractible.
\item
For the genus-$2$ splittings of lens spaces $L(p,q)$ with $1 \le q \le p/2$, if $p \equiv \pm 1 \pmod q$, then
$\mathcal P(V)$ is closed under disk surgery operation and hence it is contractible.
If $p \not\equiv \pm 1 \pmod q$, then $\mathcal P(V)$ is not weakly closed under disk surgery operation, and in fact, it is not connected.
\end{enumerate}

We remark that the weak closedness and closedness under disk surgery operation are just sufficient conditions for connectivity and contractibility respectively.
Hence it is still an open question whether the primitive disk complex $\mathcal P(V)$ in the case of $g \ge 3$ for the $3$-sphere is connected, contractible or not.

\section{Primitive curves on the boundary of a handlebody}
\label{Primitive curves on the boundary of a handlebody}

In this section, we fix a handlebody $W$ of genus $g \geq 2$,
and a {\it complete meridian system} $\{\overline{D}_1, \overline{D}_2, \ldots, \overline{D}_g \}$ for $W$.
That is $\overline{D}_1, \overline{D}_2, \ldots, \overline{D}_g$ are mutually disjoint essential disks in $W$ whose union cuts $W$ into a $3$-ball.
A simple closed curve $l$ on $\partial W$ is said to be {\it primitive} if there exists a disk $\overline{D}$ properly embedded in $W$ such that
the two simple closed curves $l$ and $\partial \overline{D}$ intersect transversely in a single point.
We call such a disk $\overline{D}$ a {\it dual disk} of $l$.

Suppose that the curve $l$ on $\partial W$ meets the union of $\partial \overline{D}_1 \cup \partial \overline{D}_2 \cup \cdots \cup \partial \overline{D}_g$
of the oriented circles transversely and minimally.
Fixing an orientation of $l$, and assigning the symbol $x_i$ to $\partial \overline{D}_i$ for each $i \in \{1,2, \ldots , g\}$,
the curve $l$ represents the conjugacy class $c(l)$ of an element of the free group $\pi_1(W)$ of rank $g$.
That is, $l$ determines a word $w$ in $\{ x_1^{\pm 1}, x_2^{\pm 1}, \ldots, x_g^{\pm 1} \}$ (up to cyclic permutation)
that can be read off from the intersections of $l$ with each of $\partial \overline{D}_i$'s.
Hence $l$ represents an element $[w]$ of the free group $\pi_1 (W) = \langle x_1, x_2 , \ldots , x_g \rangle$ (up to conjugation).
Recall that an element of a free group is said to be {\it primitive} if it is a member of some of its free basis.
If an element of a free group is primitive, then any element of its conjugacy class is also primitive.
Thus we simply say that a simple closed curve $l$ represents a primitive element of $\pi_1 (W)$ if
a member (thus every member) of $c(l)$ is primitive.
The following lemma provides a geometric interpretation of the primitive elements.

\begin{lemma}[Gordon \cite{Gordon}]
\label{lem:Gordon's criterion}
An oriented simple closed curve $l$ on $\partial W$ is primitive if and only if
$l$ represents a primitive element of $\pi_1 (W)$.
\end{lemma}

Consider the free group $F_g = \langle x_1, x_2, \ldots, x_g \rangle$ of rank $g$.
Given $1 \leq g' < g$, let $w$ be a word in $\{x_1^{\pm 1}, x_2^{\pm 1}, \ldots, x_{g'}^{\pm 1} \}$.
It is clear that if the element represented by $w$ is primitive in the free group $F_{g'} = \langle x_1, x_2, \ldots, x_{g'} \rangle$,
then so it is in $F_{g} = \langle x_1, x_2, \ldots, x_{g} \rangle$.

\begin{lemma}
\label{lem: special case}
Suppose that a word $w$ in $\{x_1^{\pm 1}, x_2^{\pm 1}, \ldots, x_{g'}^{\pm 1} \}$, where $1 \leq g' < g$, represents a primitive element of $F_g$.
If there exists an oriented simple closed curve $l$ on $\partial W$ such that
$[w] \in c(l)$ and $l \cap \overline{D}_i = \emptyset$ for each $i \in \{g'+1, \ldots,  g\}$,
then $w$ also represents a primitive element of $F_{g'}$.
\end{lemma}

\begin{proof}
Suppose that $l$ represents a primitive element of $F_g$.
By Lemma \ref{lem:Gordon's criterion}, there exists a dual disk $\overline{D}$ of $l$ in $W$.
Let $W'$ be the genus-$g'$ handlebody obtained by cutting $W$ along $\overline{D}_{g'+1} \cup \cdots \cup \overline{D}_g$.
If the disk $\overline{D}$ is disjoint from $\overline{D}_j$ for each $j \in \{g'+1, \ldots, g\}$, then $\overline{D}$ is again a dual disk of $l$ in $W'$.
Thus $l$ is a primitive curve on $\partial W'$, and so by Lemma \ref{lem:Gordon's criterion} again,
$w$ represents a primitive element of $\pi_1(W') = \langle x_1, x_2, \ldots, x_{g'} \rangle$.

If $\overline{D}$ intersects $\overline{D}_j$ for some $j \in \{g'+1, \ldots, g\}$, then
we choose an outermost subdisk of $\overline{D}_j$ cut off by $\overline{D} \cap \overline{D}_j$.
Then exactly one of the two disks, say $\overline{D}'$,
obtained by surgery on $\overline{D}$ along $\overline{D}_j$ with this outermost subdisk is again a dual disk of $l$ in $W$.
Note that $\overline{D}'$ has fewer arcs of intersection with $\overline{D}_j$ than $\overline{D}$ had.
If $\overline{D}'$ still intersects $\overline{D}_{g'+1} \cup \cdots \cup \overline{D}_g$,
we repeat this process finitely many times to obtain a dual disk $\overline{D}''$ of $l$ disjoint from $\overline{D}_j$ for each $j \in \{g'+1, \ldots, g\}$.
\end{proof}

\section{Proof of Theorem \ref{thm: main theorem}}
\label{sec: Proof of Theorem}

We first consider the genus-$3$ Heegaard splitting $(V, W; \Sigma)$ of the $3$-sphere.
Fix a complete meridian system $\{ \overline{D}_1, \overline{D}_2, \overline{D}_3 \}$ for $W$,
and assign the symbol $x_i$ to the oriented circle $\partial \overline{D}_i$ for each $i \in \{1, 2, 3\}$.
Then any oriented simple closed curve $l$ on $\partial W$ determines a word $w$ of the free group $\pi_1 (W) = \langle x_1, x_2, x_3 \rangle$ up to cyclic permutation.

Figure \ref{fig1} depicts two disks $D$ and $E$ in $V$.
The disk $E$ is the band sum of two parallel copies of the disk in Figure \ref{fig2}(a) with the ``half-twisted'' band wrapping around $\partial \overline{D}_3$ as described.
It is obvious that $D$ is a primitive disk with the dual disk $\overline {D}_2$.
The disk $E$ is also primitive by Lemma \ref{lem:Gordon's criterion} since
we read off a word determined by $\partial E$ from Figure \ref{fig1} (with a suitable choice of orientations) as
$$(x_1 x_2^{-1} x_1 x_2^{-1} x_1 x_2 x_1^{-1} x_2 x_2 x_1^{-1})
(x_1 x_2^{-1} x_2^{-1} x_1 x_2^{-1} x_1^{-1} x_2 x_1^{-1} x_2 x_1^{-1}) x_2,$$
and this word is reduced to $x_2$, representing a primitive element.

\begin{figure}[!hbt]
\centering
\includegraphics[width=9cm,clip]{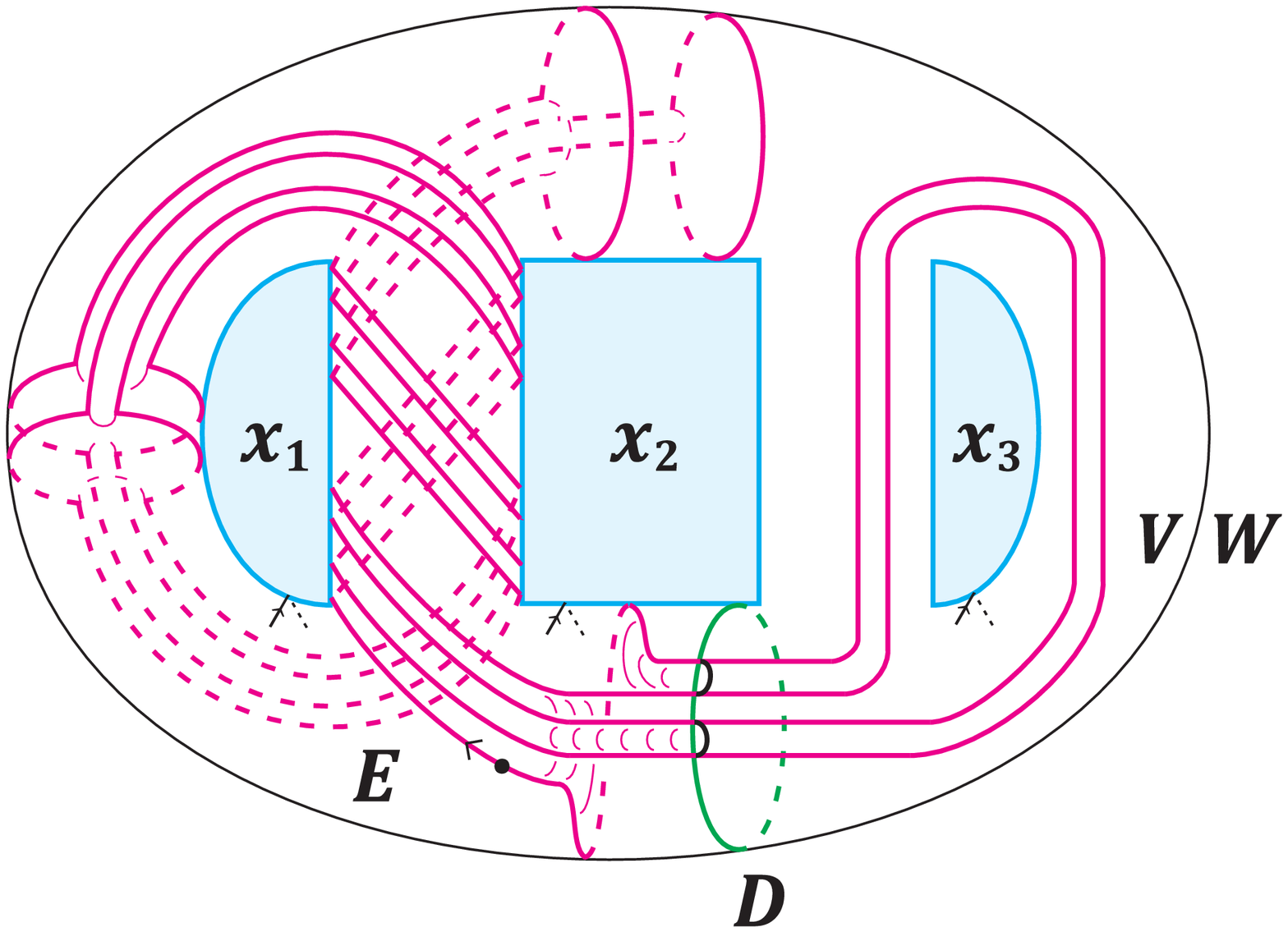}
\caption{Primitive disks $D$ and $E$ in $V$.}\label{fig1}
\end{figure}

\begin{figure}[!hbt]
\centering
\includegraphics[width=7cm,clip]{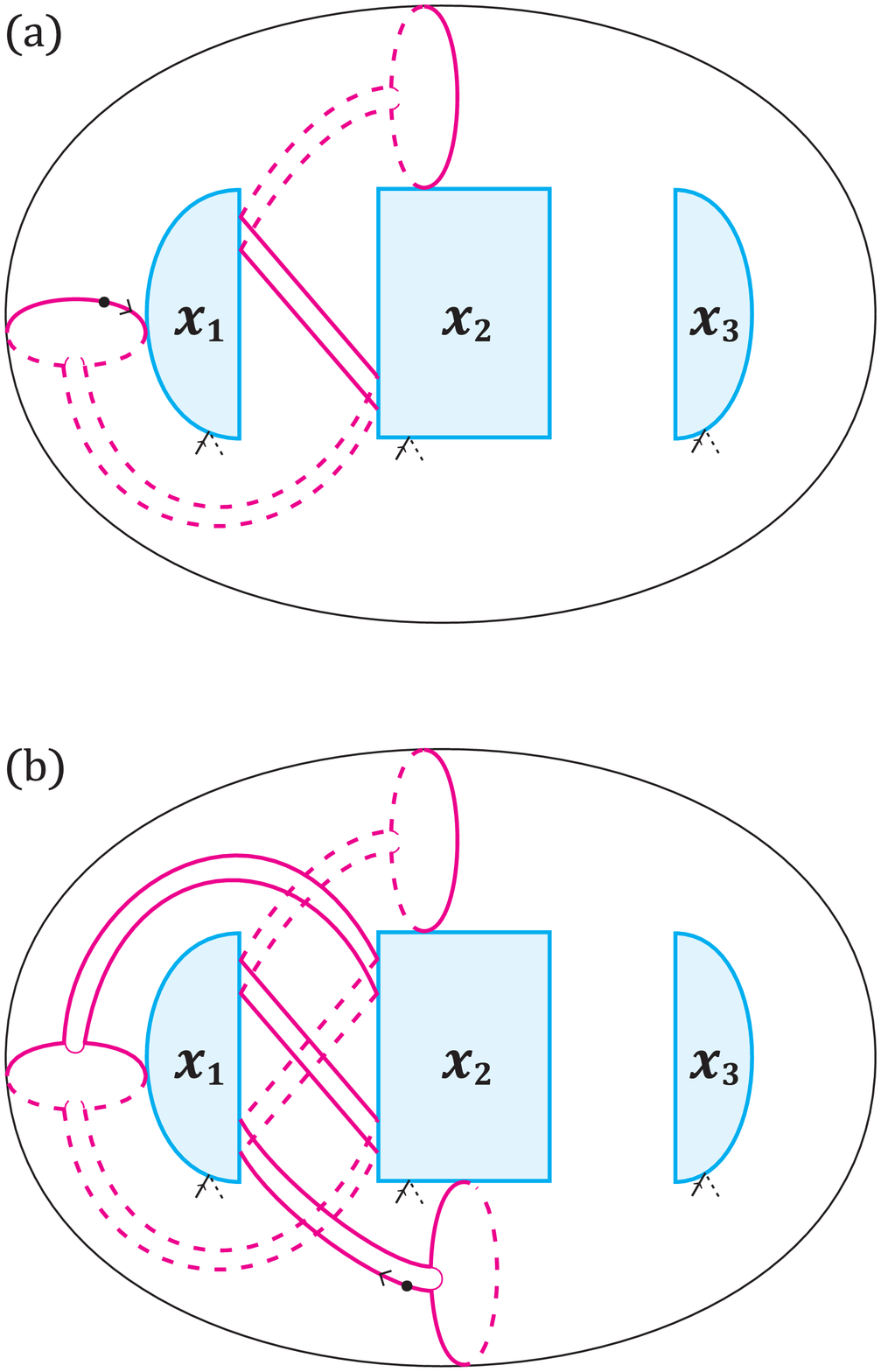}
\caption{The disks obtained by surgery.}\label{fig2}
\end{figure}

The intersection $D \cap E$ consists of two arcs.
So for each of $D$ and $E$, there are two outermost disks.
Any disk obtained by any surgery on $D$ along $E$ (and on $E$ along $D$) is one of the two disks in Figure \ref{fig2}.
The disk in Figure \ref{fig2}(a) determines a word $w_1$ of the form $x_1 x_2^{-1} x_1 x_2 x_1^{-1} x_2$,
while the disk in Figure \ref{fig2}(b) determines a word $w_2$ of the form
$x_1 x_2^{-1} x_1 x_2^{-1} x_1 x_2 x_1^{-1} x_2 x_2 x_1^{-1} x_2$.
Both disks are disjoint from $\partial \overline{D}_3$ and hence the generator $x_3$ does not appear in both $w_1$ and $w_2$.
So $w_1$ and $w_2$ represent elements of the free group $\langle x_1, x_2 \rangle$.
We observe that each of $w_1$ and $w_2$ is cyclically reduced and contains $x_1$ and $x_1^{-1}$ simultaneously (also $x_2$ and $x_2^{-1}$ simultaneously).
Thus, by Osborne-Zieschang \cite{OZ81}, the elements represented by $w_1$ and $w_2$ are not primitive in the free group $\langle x_1, x_2 \rangle$, and
hence are not primitive in the free group $\langle x_1, x_2, x_3 \rangle$ by Lemma \ref{lem: special case}.
Thus the two disks in Figure \ref{fig2} are not primitive in $V$ by Lemma \ref{lem:Gordon's criterion}.

So far we gave an example for the genus-$3$ Heegaard splitting of the $3$-sphere,
but the same argument applies for any genus $g$ with $g \ge 3$.
See Figure \ref{fig3}.

\begin{figure}[!hbt]
\centering
\includegraphics[width=9cm,clip]{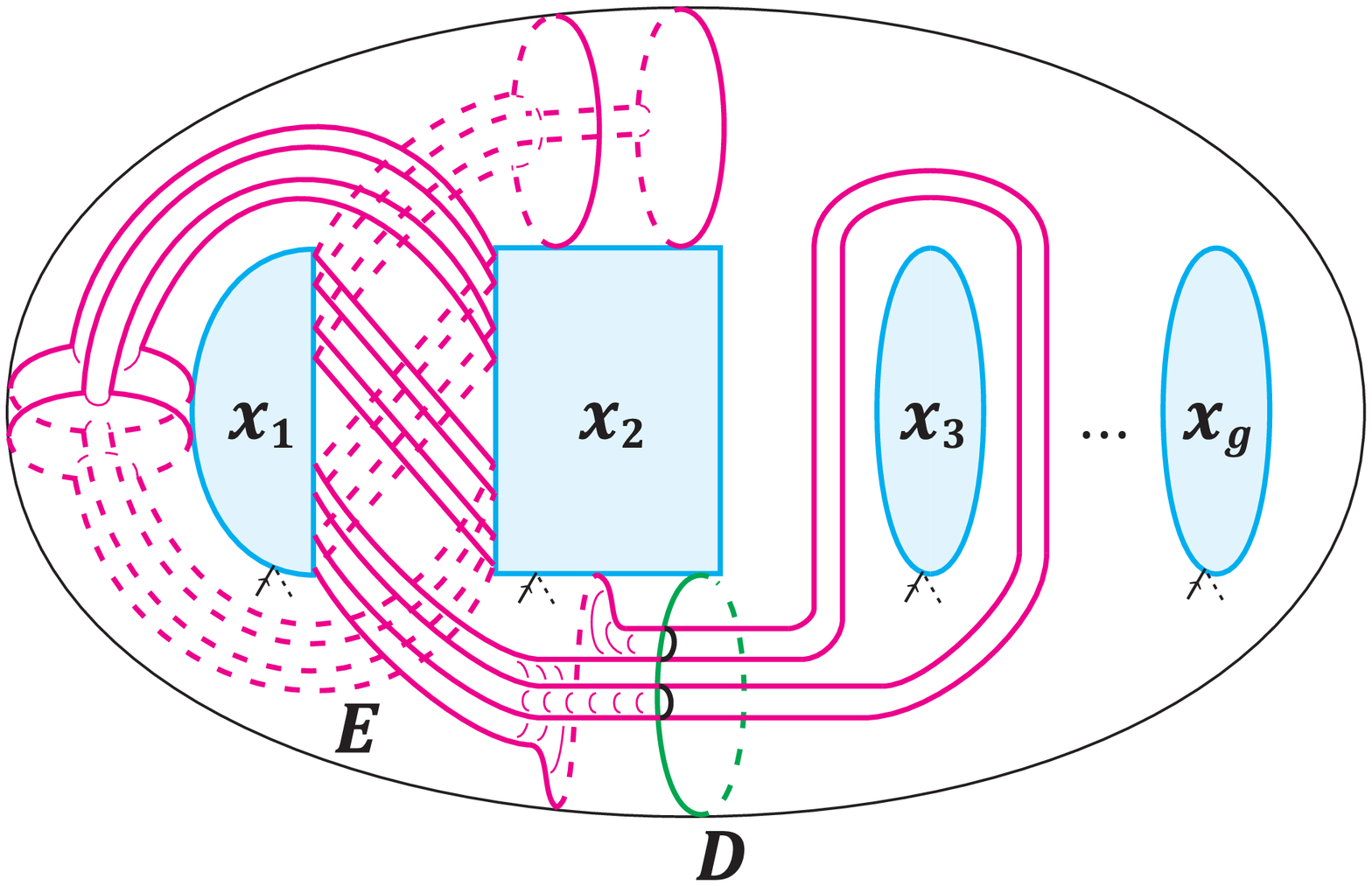}
\caption{Primitive disks $D$ and $E$ for the case of genus $g \ge 3$.}
\label{fig3}
\end{figure}

\vspace{0.1cm}

{\noindent \bf Acknowledgments.}
The first-named author was supported by the Basic Science Research Program through the National Research Foundation of Korea (NRF)
funded by the Ministry of Education (NRF-201800000001768).
The second-named author is supported in part by JSPS KAKENHI Grant Numbers 15H03620, 17K05254, 17H06463, and JST CREST Grant Number JPMJCR17J4.
The third-named author was supported by the Basic Science Research Program through the National Research Foundation of Korea (NRF)
funded by the Ministry of Education (2018R1D1A1A09081849).

\end{document}